\documentclass{amsart}
\usepackage{amsmath}
\usepackage{amsthm}
\usepackage{amsfonts}
\usepackage{amssymb}

\usepackage[margin=1.55in]{geometry}
\usepackage{verbatim}
\newtheorem{theorem}{Theorem}

\newtheorem*{question}{Question}

\newtheorem{prop}[theorem]{Proposition}
\theoremstyle{definition}

\usepackage{url}
\usepackage{color}

\newcommand\nc{\newcommand}
\nc\qf[1]{\bigl<#1\bigr>}
\nc\qfset[1]{\mathcal{#1}}
\nc\card[1]{\left\vert #1\right\vert}
\nc\s{\qfset{S}}
\nc\marker{*}
\nc\defeq{:=}
\nc\ZZ{{\mathbb{Z}}}
\nc\NN{{\mathbb{N}}}
\nc\QQ{{\mathbb{Q}}}
\nc\ra{{\rightarrow}}
\nc\Llra{{\Longleftrightarrow}}
\nc\Union{{\mathop{U}\nolimits}}
\nc\Prod{{\mathop{\sf P}\nolimits}}

\title[Minimal $\s$-Universality Criteria May Vary in Size]{Minimal $\s$-Universality Criteria\\May Vary in Size}

\author{Noam D. Elkies}
\address{Department of Mathematics, Harvard University\newline\indent
   One Oxford Street\newline\indent Cambridge, MA 02138}
\email{elkies@math.harvard.edu}

\author{Daniel M. Kane}
\address{Department of Mathematics, Harvard University\newline\indent
   One Oxford Street\newline\indent Cambridge, MA 02138}
\email{dankane@math.harvard.edu, aladkeenin@gmail.com}

\author{Scott Duke Kominers}
\address{Department of Economics, Harvard University,
  and Harvard Business School\newline \indent
  Wyss Hall, Harvard Business School\newline\indent
  Soldiers Field, Boston, MA 02163 }
\email{kominers@fas.harvard.edu, skominers@gmail.com}


\begin{document}
\begin{abstract}
In this note, we give simple examples of sets $\s$ of quadratic
forms that have minimal $\s$-universality criteria of multiple cardinalities.
This answers a question of Kim, Kim, and Oh~\cite{Kim:finite} in
the negative.
\end{abstract}
\maketitle

A quadratic form $Q$ \textit{represents} another quadratic form $L$
if there exists a $\ZZ$-linear, bilinear form-preserving
injection $L \to Q$.  In this note, we consider only positive-definite
quadratic forms, and assume unless stated otherwise that every form
is classically integral (equivalently:~has a Gram matrix
with integer entries).  For a set $\s$ of such forms, a quadratic
form is called \textit{(classically) $\s$-universal} if it represents
all quadratic forms in $\s$.

Denote by $\NN$ the set $\{1,2,3,\ldots\}$ of natural numbers.
In 1993, Conway and Schneeberger (see~\cite{Bhargava:Fif,Conway:universality})
proved the ``Fifteen Theorem'':
$\{ax^2:a\in\NN\}$-universal forms can be exactly characterized
as the set of forms which represent all of the forms in the finite set
$\{x^2,2x^2,3x^2,5x^2,6x^2,7x^2,10x^2,14x^2,15x^2\}$.
This set is thus said to be a ``criterion set'' for $\{ax^2:a\in\NN\}$.
In general, for a set $\s$ of quadratic forms of bounded rank,
a form~$Q$\/ is said to be \emph{$\s$-universal}\/
if it represents every form in~$\s$;
an \emph{$\s$-criterion set} is a subset $\s_\marker\subset \s$
such that every $\s_\marker$-universal form is $\s$-universal.  
Following the Fifteen Theorem, Kim, Kim, and Oh~\cite{Kim:finite} proved
that, surprisingly, finite $\s$-universality criteria exist in general. 
\begin{theorem}[Kim, Kim, and Oh~\cite{Kim:finite}]
Let $\s$ be any set of quadratic forms of bounded rank.  Then, there
exists a finite $\s$-criterion set.
\end{theorem}
Kim, Kim, and Oh~\cite{Kim:finite} observed that there
may be multiple $\s$-criterion sets $\s_\marker\subset\s$ which
are \emph{minimal} in the sense that for each $L\in \s_\marker$
there exists a $Q$ that is $(\s_\marker\setminus\{L\})$-universal
but not $\s$-universal.\footnote{Kim, Kim, and Oh~\cite{Kim:finite}
  gave a simple example of a set of quadratic forms $\s$ with
  multiple minimal $\s$-criterion sets:
  $$
  \s=\left\{\qf{2^i}\oplus\qf{2^j}\oplus\qf{2^k}:0\leq
  i,j,k\in \ZZ\right\},
  $$
  which has $\s$-criterion sets
  $\left\{\qf{1}\oplus\qf{1}\oplus\qf{1},\qf{1}\oplus\qf{1}\oplus\qf{2}\right\}$
  and
  $\left\{\qf{1}\oplus\qf{1}\oplus\qf{1},\qf{2}\oplus\qf{2}\oplus\qf{2}\right\}$.
  }
Given this observation, they asked (and speculated to be difficult)
the following question:
\begin{question}[Kim, Kim, and Oh~\cite{Kim:finite}; Kim~\cite{kim2004recent}]
Is it the case that for all sets $\s$ of quadratic forms (of bounded
rank), all minimal $\s$-criterion sets have the same cardinality?
Formally, is $$\card{\s_\marker}= \card{\s_\marker'}$$ for all
minimal $\s$-criterion sets $\s_\marker$ and $\s_\marker'$?
\end{question}
In this brief note, we give simple examples that answer this
question in the negative.
In each case we choose some quadratic form~$A$,
and let $\s$ be the set of quadratic forms represented by~$A$,
so that $\s_\marker = \{A\}$ is a minimal $\s$-criterion set. 
We then exhibit one or more $\s_\marker'\subset \s$ that are 
finite but of cardinality $2$ or higher,
and prove that $\s_\marker'$ is also a minimal $\s$-criterion set.

We first give an example where $A$ is diagonal of rank~$3$
and $\s_\marker'$ consists of
one diagonal form of rank~$2$ and one of rank~$3$.
We then give even simpler examples of higher rank where each
$L \in \s_\marker'$ has rank smaller than that of~$A$,
often with $A = \oplus_{L \in \s_\marker'} L$.

It will at times be convenient to switch from the terminology of
quadratic forms to the equivalent notions for lattices;
we shall do this henceforth without further comment.
For example we identify the form $\qf{1}$ with the lattice~$\ZZ$.

\section*{An Example of Rank $3$}

Let $A\defeq\qf{1}\oplus \qf{1} \oplus \qf{2}$, be the quadratic
form that is the orthogonal direct sum of two copies of the form
$\qf{1}$ and one copy of the form $\qf{2}$.
Let $B\defeq\qf{1}\oplus \qf{1}$ and
$C\defeq\qf{2}\oplus \qf{2}\oplus \qf{2}$.
Let $\s$ be the set of quadratic forms represented by~$A$.

\begin{theorem}\label{thm:the}
Both $\{A\}$ and $\{B,C\}$ are minimal $\s$-criterion sets.
\end{theorem}

Theorem~\ref{thm:the} provides an example of two minimal
$\s$-criterion sets of different cardinalities.

\begin{proof}[Proof of Theorem~\ref{thm:the}]
Clearly, $\{A\}$ is a minimal $\s$-criterion set.  Moreover, it
is clear that while $B,C\in \s$, neither $\{B\}$ nor $\{C\}$ is
an $\s$-criterion set since neither $B$ nor $C$ can embed~$A$.
It therefore only remains to show that $\{B,C\}$ is an $\s$-criterion
set.  To show this, it suffices to prove that any quadratic form $Q$\/
that represents both $B$ and $C$ also represents~$A$.

First, we note that any vector $v$ of norm $2$ in an integer-matrix
quadratic form $Q$ which is not a sum of two orthogonal $Q$-vectors
of norm $1$ must be orthogonal to all $Q$-vectors of norm $1$.
Indeed, if $v,w\in Q$, $(v,v)=2$, $(w,w)=1$, and $(v,w)\neq 0$, then
we may assume that $(v,w)=1$
(by Cauchy-Schwarz, $(v,w)$ is either $1$ or $-1$, and in the
latter case we may replace $w$ by $-w$).  Then $v=w+(v-w)$, where
$w$ and $v-w$ are orthogonal vectors of norm $1$.

Suppose for sake of contradiction that $Q$ is a quadratic form
that represents $B$ and $C$ but not $A$.  Since $Q$ represents
$B$ but not $A$, there is no norm-$2$ vector of $Q$ orthogonal to
all norm-$1$ vectors of $Q$.  Since $Q$ represents $C$, it must
contain three orthogonal norm-$2$ vectors, $u$, $v$, and $w$.
By the above observation, we may write $u$ as a sum of norm-$1$ vectors,
say $u=x+y$ for some orthogonal norm-$1$ vectors $x,y\in Q$.

Now, each of $v$ and $w$ is orthogonal to $u$ but not orthogonal
to both $x$ and $y$ (since otherwise we could embed $A$ as the span
of $\{x,y,v\}$ or $\{x,y,w\}$).  We claim that this implies that
both $v$ and $w$ are of the form $\pm(x-y)$:  Since $v$ is not orthogonal
to both $x$ and $y$, we may assume without loss of generality that
$v$ is not orthogonal to $x$.  Perhaps replacing $v$ with $-v$,
we may assume that $(v,x)=1$.  We then have $v=x+z$ for some
unit vector $z$ orthogonal to $x$.  We have
$$
0=(u,v) = (x+y,x+z) = (x,x) + (x,z) + (y,x) + (y,z) = 1 + (y,z),
$$
hence $(y,z)=-1$.  Since both $y$ and $z$ are unit vectors, this
implies that $z=-y$, hence $v=x-y$.  An analogous argument shows
that $w$ is of the form $\pm(x-y)$.

Finally, if both $v$ and $w$ are of the form $\pm(x-y)$, then
$(v,w) \in \{2, -2\}$,
contradicting the fact that $v$ and $w$ are orthogonal.\end{proof}

\section*{Examples of Higher Rank}

We begin with a simple example in rank~$9$. We give two proofs of
the correctness of this example,
each of which suggests a different generalization.

\begin{prop}\label{thm:prop1}
Let $A = E_8 \oplus \ZZ$, and let
$\s$ be the set of quadratic forms represented by~$A$.
Then both $\{A\}$ and $\{E_8,\ZZ\}$ are minimal $\s$-criterion sets.
\end{prop}

\begin{proof}
As in the proof of Theorem~\ref{thm:the}, we need only prove that
any quadratic form~$Q$\/ that represents both $E_8$ and $\ZZ$
also represents $E_8 \oplus \ZZ$.

\emph{First argument.}  Fix a copy of~$E_8$ in~$Q$.  Choose any 
copy of~$\ZZ$ in~$Q$, that is,  any vector~$v\in Q$ with $(v,v)=1$.
Let $\pi: Q \ra E_8 \otimes \QQ$ be orthogonal projection.  Then,
$(\pi(v),w) = (v,w) \in \ZZ$ for all $w \in E_8$, so $\pi(v) \in E_8^*$.
But $E_8$ is self-dual, and has minimal norm~$2$.  Since
$(\pi(v),\pi(v)) \leq (v,v)$, it follows that $\pi(v)=0$,
that is, $v$ is orthogonal to $E_8$.  Hence $Q$ contains $E_8 \oplus \ZZ$
as claimed.

\emph{Second argument.}  Since $E_8$ and $\ZZ$ are unimodular,
they are direct summands of~$Q$ (again because $\pi(v) \in E_8$ for all
$v \in Q$, and likewise for the projection to $\ZZ \otimes \QQ$).
But $E_8$ and $\ZZ$ are indecomposable, and any positive-definite
lattice is uniquely the direct sum of indecomposable summands.
Hence $Q = \oplus_k Q_k$ for some indecomposable $Q_k \subset Q$,
which include $E_8$ and $\ZZ$, so again we conclude that
$Q$ represents $E_8 \oplus \ZZ$.
\end{proof}

The first argument for Proposition~\ref{thm:prop1} generalizes as follows.

\begin{prop}\label{thm:prop2}
Let $A = L \oplus L'$, where $L'$ is generated by vectors $v_i$
of norms $(v_i,v_i)$ less than the minimal norm of vectors in
the dual lattice\footnote{
  This dual lattice is the only lattice we consider
  that might fail to be classically integral.
  }~$L^*$.
Let $\s$ be the set of quadratic forms represented by~$A$.
Then, both $\{A\}$ and $\{L,L'\}$ are minimal $\s$-criterion sets.
\end{prop}

\begin{proof}
As before, it is enough to show that if $Q$\/ represents both
$L$ and~$L'$ then it represents $L \oplus L'$.
Let $\pi$ be the orthogonal projection to $L \otimes \QQ$.
Then $\pi(v_i) \in L^*$ for each~$i$, whence $\pi(v_i) = 0$ because 
$$(\pi(v_i),\pi(v_i))\leq (v_i,v_i)<\min_{{v\in L^*}\atop{v\neq 0}}(v,v).$$ 
Thus, the copy of~$L'$ generated by the $v_i$ is orthogonal to~$L$. 
This gives the desired representation of $L \oplus L'$ by~$Q$.
\end{proof}

\emph{Examples.}  We may take $L' = \ZZ^n$
for any $n \in \NN$, and $L \in\{ E_6, E_7, E_8\}$;
choosing $L=E_6$ and $n=1$ gives an example of rank~$7$,
the smallest we have found with this technique.  We may also
take $L$ to be the Leech lattice; then $L'$ can be any lattice
generated by its vectors of norms~$1$, $2$, and~$3$.  
There are even examples with neither $L$ nor $L'$ unimodular.
Indeed, such examples may have arbitrarily large discriminants.
For instance, let $\Lambda_{23}$ be the laminated lattice of rank~$23$
(the intersection of the Leech lattice with the orthogonal complement of
one of its minimal vectors);
this is a lattice of discriminant~$4$ and minimal dual norm~$3$.
So we can take $L = \Lambda_{23}^n$ for arbitrary
$n \in \NN$, and choose any root lattice for~$L'$.

$\phantom{kludge}$

The second argument for Proposition~\ref{thm:prop1}
generalizes in a different direction.  We use the following notations.
For a collection $\Pi$ of sets, let
$\Union(\Pi)$ be their union $\cup_{\qfset{P}\in\Pi} \qfset{P}$;
and for a finite set $\qfset{P}$ of lattices,
let $\Prod(\qfset{P})$ be the direct sum
$\oplus_{L \in \qfset{P}} L$.  Say that two lattices $L,L'$
are \emph{coprime} if they have no indecomposable summands in common.

\begin{prop}\label{thm:prop3}
Let $A = \Prod(\qfset{P})$, where
$\qfset{P}$ is a finite set of pairwise coprime, unimodular lattices;
and let $\Pi$ be a family of subsets of~$\qfset{P}$
such that $\Union(\Pi) = \qfset{P}$.
Then $\s'_\marker \defeq \{\Prod(\qfset{R}) : \qfset{R} \in \Pi\}$
is an $\s$-criterion set for the set~$\s$ of
quadratic forms represented by~$A$.
Moreover, $\s'_\marker$ is a minimal $\s$-criterion set
if and only if $\Union(\Pi\setminus \{\qfset{R}\})$
is smaller than~$\qfset{P}$ for each $\qfset{R} \in \Pi$.
\end{prop}

\begin{proof}
We repeatedly apply the observation that
if $\qfset{P}$ is a set of pairwise coprime lattices,
each of which is a direct summand of a lattice~$Q$,
then $\Prod(\qfset{P})$ is also a direct summand of~$Q$.
Since any unimodular sublattice of an integer-matrix lattice
is a direct summand, it follows that $Q$\/ represents
$\Prod(\qfset{R})$ for each $\qfset{R} \in \Pi$
$\Llra$ $Q$\/ represents each lattice in $\Union(\Pi) = \qfset{P}$
$\Llra$ $Q$\/ represents $\Prod(\qfset{P}) = A$.
That is, $\s'_\marker$ is a criterion set for~$A$.
Moreover, replacing $\Pi$ by any subset $\Pi' = \Pi \setminus \qfset{R}$
shows that $\{\Prod(\qfset{R}) : \qfset{R} \in \Pi'\}$
is a criterion set for $\Prod(\Union(\Pi'))$.
Thus $\s'_\marker$ is minimal if and only if
$\Union(\Pi \setminus \qfset{R}) \subsetneq \qfset{P}$
for each $\qfset{R} \in \Pi$.
\end{proof}

\emph{Examples.}  We may take for $\Pi$ any partition of~$\qfset{P}$,
and then $A = \Prod(\s_\marker') = \oplus_{L \in \s_\marker'} L$.
Proposition~\ref{thm:prop1} is
the special case $\qfset{P} = \{E_8, \ZZ\}$, $\Pi = \{\{E_8\}, \{\ZZ\}\}$.
(The similar case $\qfset{P} = \{E_8, \ZZ^8\}$, $\Pi = \{\{E_8\}, \{\ZZ^8\}\}$
was in effect used already by Oh \cite[Theorem 3.1]{Oh:MinRank} and the
third author \cite{Kom2} in the study of $8$-universality criteria.)
Since $\card{\qfset{P}}$ can be any natural number~$n$,
Proposition~\ref{thm:prop3} produces for each~$n$\/
a lattice~$A$ for which $\s$ has minimal criterion sets
of (at least) $n$\/ distinct cardinalities.

\section*{Remarks}

The examples presented here show that minimal $\s$-criterion sets
may vary in size.  Further examples can be obtained by mixing
the techniques of Theorem~\ref{thm:the} and
Propositions~\ref{thm:prop2} and~\ref{thm:prop3};
for instance,
$\{ \qf{1}\oplus \qf{1} \oplus \qf{2} \oplus E_8 \oplus \Lambda_{23} \}$
and
$\{
\qf{1}\oplus \qf{1}, \qf{2}\oplus \qf{2}\oplus \qf{2} \oplus E_8, \Lambda_{23}
\}$
are both minimal criterion sets for the set of lattices represented by
$\qf{1}\oplus \qf{1} \oplus \qf{2} \oplus E_8 \oplus \Lambda_{23}$.
However, it is unclear (and appears difficult
to characterize in general) for which $\s$ this phenomenon occurs.

For the sets $\s_n$ of rank-$n$ quadratic forms, criterion sets
are known only in the cases $n=1,2,8$ (see \cite{Bhargava:Fif,Conway:universality}, \cite{Kim:universal}, and \cite{Oh:MinRank}, respectively).   Few
criterion sets beyond those for $\s_n$ ($n=1,2,8$) have been explicitly
computed.

Meanwhile, in the cases $n=1,2,8$, the minimal $\s_n$-criterion
sets are known to be \emph{unique} (see~\cite{kim2004recent}, \cite{Kom},
and \cite{Kom2}), in which case the answer to the question we examine
is (trivially) affirmative.  But there is not yet a general characterization
of the~$\s$ that have unique minimal $\s$-criterion sets
(see~\cite{kim2004recent}).
It seems likely that such a result would be essential in making
progress towards a general answer to the question of Kim, Kim, and
Oh~\cite{Kim:finite} that we studied here.

\section*{Acknowledgements}

While working on this paper, Elkies was supported in part by NSF grant
DMS-0501029, and Kane and Kominers were supported in part by
NSF Graduate Research Fellowships.

\bibliographystyle{amsalpha}
\bibliography{qf}
\end{document}